\documentclass[
final
 , nomarks
]{dmtcs-episciences}

\usepackage[utf8]{inputenc}
\usepackage{subfigure}

\usepackage{enumerate}
\usepackage{xcolor}
\usepackage{amssymb,amsmath,amsfonts}
\usepackage{comment}
\usepackage[english]{babel}
\usepackage[absolute]{textpos}
\usepackage{enumitem}
\usepackage{mathtools}
\usepackage{times}

\def\N{\mathbb{N}}

\def\O{\mathcal{O}}
\def\C{\mathcal{C}}
\def\I{\mathcal{I}}
\def\Ff{\mathcal{F}}
\def\Bb{\mathcal{B}}
\DeclareMathOperator{\ctw}{\mathbf{ctw}}
\DeclareMathOperator{\pw}{\mathbf{pw}}
\def\barcs{\overleftarrow{A}}

\DeclareMathOperator{\first}{\mathsf{start}}
\DeclareMathOperator{\last}{\mathsf{end}}

\DeclareMathOperator{\cut}{\mathsf{cut}}
\DeclareMathOperator{\vcut}{\mathsf{vcut}}

\newcommand{\Pow}{\mathsf{Pow}}
\newcommand{\Comps}{\mathsf{Comps}}
\newcommand{\wh}[1]{\widehat{#1}}
\newcommand{\whH}{\wh{H}}
\newcommand{\Ss}{\mathcal{S}}

\usepackage[amsmath,thmmarks,hyperref]{ntheorem}
\usepackage{cleveref}

\theoremnumbering{arabic}
\theoremstyle{plain}
\theoremsymbol{}
\theorembodyfont{\itshape}
\theoremheaderfont{\normalfont\bfseries}
\theoremseparator{.}

\newtheorem{theorem}{Theorem}
\crefformat{theorem}{#2Theorem~#1#3}
\Crefformat{theorem}{#2Theorem~#1#3}

\newcommand{\newtheoremwithcrefformat}[2]{%
  \newtheorem{#1}[theorem]{#2}%
  \crefformat{#1}{##2\MakeUppercase#1~##1##3}%
  \Crefformat{#1}{##2\MakeUppercase#1~##1##3}%
}
\newcommand{\newseptheoremwithcrefformat}[2]{%
  \newtheorem{#1}{#2}%
  \crefformat{#1}{##2\MakeUppercase#1~##1##3}%
  \Crefformat{#1}{##2\MakeUppercase#1~##1##3}%
}

\newtheoremwithcrefformat{lemma}{Lemma}
\newtheoremwithcrefformat{proposition}{Proposition}
\newtheoremwithcrefformat{observation}{Observation}
\newtheoremwithcrefformat{corollary}{Corollary}
\newseptheoremwithcrefformat{claim}{Claim}
\newseptheoremwithcrefformat{definition}{Definition}
\theorembodyfont{\upshape}
\newtheoremwithcrefformat{example}{Example}
\newseptheoremwithcrefformat{remark}{Remark}

\theoremstyle{nonumberplain}
\newseptheoremwithcrefformat{conjecture}{Conjecture}
\theoremheaderfont{\scshape}
\theorembodyfont{\normalfont}
\theoremsymbol{\ensuremath{\square}}
\newtheorem{proof}{Proof}

\theoremsymbol{\ensuremath{\lrcorner}}
\newtheorem{clproof}{Proof}

\def\cqedsymbol{\ifmmode$\lrcorner$\else{\unskip\nobreak\hfil
\penalty50\hskip1em\null\nobreak\hfil$\lrcorner$
\parfillskip=0pt\finalhyphendemerits=0\endgraf}\fi}

\title{On the Erdős-Pósa property for immersions and topological minors in tournaments\thanks{This work is 
a part of projects that have received funding from the European Research Council (ERC) 
under the European Union's Horizon 2020 research and innovation programme, grant agreements No.~714704 (\L{}.~Bo\.zyk) and No.~677651 (Mi.~Pilipczuk).
}}

\author{\L{}ukasz Bo\.zyk
\and Micha\l{}~Pilipczuk}

\affiliation{
  Institute of Informatics, University of Warsaw, Poland}
\received{2021-01-19}
\revised{2022-03-08}
\accepted{2022-03-09}
\begin{document}
\publicationdetails{24}{2022}{1}{12}{7099}

\begin{textblock}{20}(0.5, 11.6)
\includegraphics[width=40px]{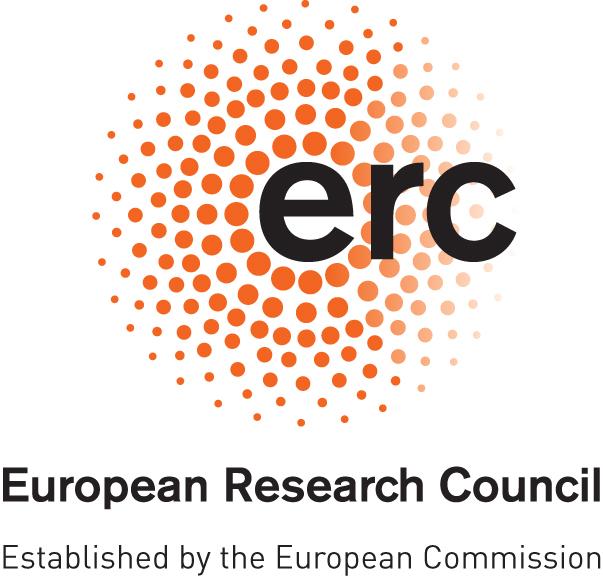}%
\end{textblock}
\begin{textblock}{20}(0.25, 12)
\includegraphics[width=60px]{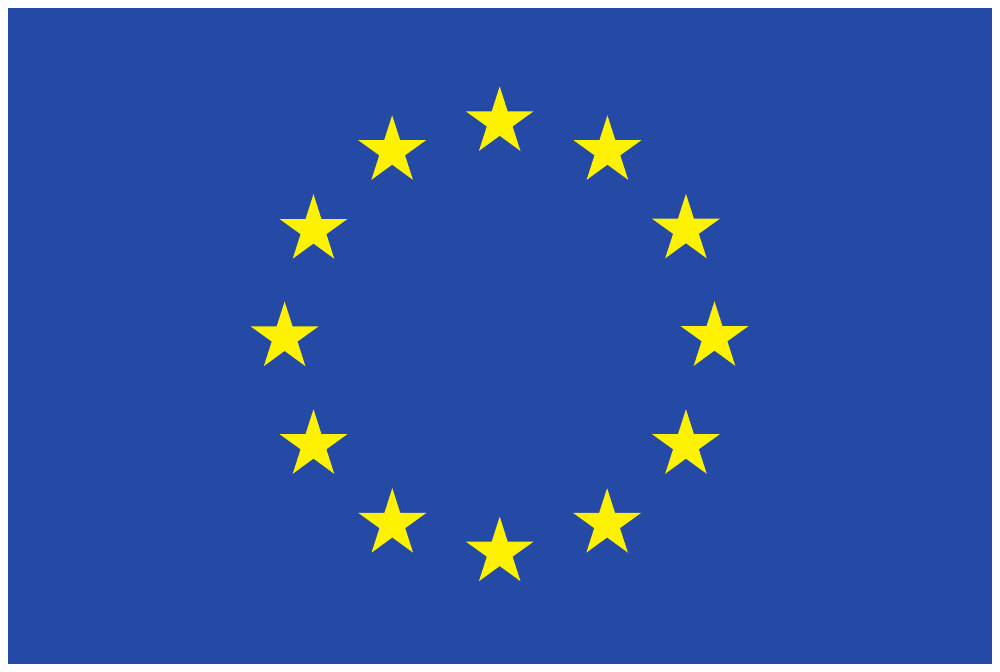}%
\end{textblock}

\maketitle

\begin{abstract}
We consider the Erdős-Pósa property for immersions and topological minors in tournaments.
We prove that for every simple digraph $H$, $k\in \N$, and tournament $T$, the following statements~hold:
\begin{itemize}
 \item If in $T$ one cannot find $k$ arc-disjoint immersion copies of $H$, then there exists a set of $\O_H(k^3)$ arcs that intersects all immersion copies of $H$ in $T$.
 \item If in $T$ one cannot find  $k$ vertex-disjoint topological minor copies of $H$, then there exists a set of $\O_H(k\log k)$ vertices that intersects all topological minor copies of $H$ in $T$.
\end{itemize}
This improves the results of Raymond [DMTCS '18], who proved similar statements under the assumption that $H$ is strongly connected.
\end{abstract}

\section{Introduction}

The Erd\H{o}s-P\'osa problems concern hitting-packing duality in set systems arising in different combinatorial settings. Suppose we consider a universe $U$ and a family $\Ss$ of subsets of this universe. The {\emph{packing number}} of the set system $(U,\Ss)$ is the maximum number of disjoint sets that one can find in $\Ss$, while the {\emph{hitting number}} is the minimum size of a subset of $U$ that intersects every set in $\Ss$. Clearly, the packing number is always a lower bound for the hitting number. In general, we cannot hope for the reverse inequality to hold even in the following weak sense: we would wish that the hitting number is bounded by a function of the packing number. However, such a bound often can be established when we have further assumptions on the origin of the set system $(U,\Ss)$, e.g., $\Ss$ comprises of some well-behaved combinatorial objects.

The first result of this kind was delivered by Erd\H{o}s and P\'osa~\cite{EP}, who proved that for every undirected graph $G$ and $k\in \N$, one can find in $G$ either $k$ vertex-disjoint cycles, or a set of $\O(k\log k)$ vertices that meets all the cycles. This idea can be generalized to packing and hitting minor models in graphs in the following sense. Consider any fixed undirected graph $H$. We say that $H$ has the {\emph{Erd\H{o}s-P\'osa property for minors}} if there exists a function $f$ such that for every graph $G$ and $k\in N$, one can find in $G$ either $k$ vertex-disjoint minor models of $H$, or a set of at most $f(k)$ vertices that meets all minor models of $H$. Thus, the result Erd\H{o}s and P\'osa asserts that the triangle $K_3$ has the Erd\H{o}s-P\'osa property for minors. Robertson and Seymour~\cite{RobertsonS86} proved that a graph $H$ has the Erd\H{o}s-P\'osa property for minors if and only if $H$ is planar.

Since the work of Erd\H{o}s and P\'osa, establishing the Erd\H{o}s-P\'osa property for different objects in graphs, as well as finding tight estimates on the best possible {\emph{bounding functions}} $f$, became a recurrent topic in graph theory. There are still many open problems in this area. For instance, the setting of directed graphs ({\emph{digraphs}}) remains rather scarcely explored. The analogue of the result of Erd\H{o}s and P\'osa for packing and hitting directed cycles was finally established by Reed et al.~\cite{ReedRST96} after functioning for over $20$ years as the {\emph{Younger's conjecture}}, while a characterization of strongly connected digraphs $H$ posessing the Erd\H{o}s-P\'osa property for topological minors was recently announced by Amiri et al.~\cite{AmiriKKW16}. We refer the reader to the survey of Raymond and Thilikos~\cite{RaymondT17} and to a website maintained by Raymond~\cite{jfr} for an overview of the current state of knowledge on Erd\H{o}s-P\'osa problems.

\paragraph*{Our contribution.} We consider the Erd\H{o}s-P\'osa problems for immersions and topological minors when the host graph $G$ is restricted to be a {\emph{tournament}}: a directed graph where every pair of vertices is connected by exactly one arc. Recall here that a directed graph $H$ can be {\emph{immersed}} in a digraph $D$ if one can find a mapping that maps vertices of $H$ to pairwise different vertices of $D$, and arcs of $H$ to pairwise arc-disjoint oriented paths in $H$ connecting the images of endpoints. The subgraph of $D$ consisting of all the vertices and arcs participating in the image of the mapping is called an {\emph{immersion copy}} of $H$ in $D$. We define topological minors and topological minor copies in the same way, except that we require the paths to be vertex-disjoint instead of arc-disjoint. See Section~\ref{sec:prelims} for formal definitions.

As usual, the Erd\H{o}s-P\'osa property for topological minors refers to packing vertex-disjoint topological minor copies and hitting topological minor copies with vertices. Since the notion of an immersion is based on arc-disjointness, it is more natural to speak about packing arc-disjoint immersion copies and hitting immersion copies with arcs instead of vertices. The following two definitions  formally introduce the properties we are interested in.

\begin{definition}
A directed graph $H$ has the \emph{Erd\H{o}s-P\'osa property for immersions in tournaments} if there is a function $f\colon\N\to\N$, called further a {\emph{bounding function}}, such that for every $k\in\N$ and every tournament~$T$, at least one of the following~holds:
\begin{itemize}[nosep]
\item $T$ contains $k$ pairwise arc-disjoint immersion copies of $H$; or
\item there exists a set of at most $f(k)$ arcs of $T$ that intersects all immersion copies of $H$ in $T$
\end{itemize} 
\end{definition}

\begin{definition}
A directed graph $H$ has the \emph{Erd\H{o}s-P\'osa property for topological minors in tournaments} if there is a function $f\colon\N\to\N$, called further a {\emph{bounding function}}, such that for every $k\in\N$ and every tournament $T$, at least one of the following~holds:
\begin{itemize}[nosep]
\item $T$ contains $k$ pairwise vertex-disjoint topological minor copies of $H$; or
\item there exists a set of at most $f(k)$ vertices of $T$ that intersects all topological minor copies of $H$ in $T$.
\end{itemize} 
\end{definition}

These two properties were investigated by Raymond~\cite{R18}, who proved that as long as $H$ is {\emph{simple}} --- there are no multiple arcs with the same head and tail --- and {\emph{strongly connected}} --- for every pair of vertices $u,v$, there are directed paths both from $u$ to $v$ and from $v$ to $u$ --- the considered Erd\H{o}s-P\'osa properties~hold.

\begin{theorem}[Theorem~2 of~\cite{R18}]
 Every simple, strongly connected directed graph has the Erd\H{o}s-P\'osa property for topological minors in tournaments.
\end{theorem}

\begin{theorem}[Theorem~3 of~\cite{R18}]
 Every simple, strongly connected directed graph has the Erd\H{o}s-P\'osa property for immersions in tournaments.
\end{theorem}

Raymond asked in~\cite{R18} whether the assumption that the digraph in question is strongly connected can be dropped, as it was important in his proof. We answer this question in affirmative by proving the~following.

\begin{theorem}\label{thm:main-imm}
Every simple directed graph $H$ has the Erd\H{o}s-P\'osa property for immersions in tournaments with bounding function $f(k)\in \O_H(k^3)$. 
\end{theorem}

\begin{theorem}\label{thm:main-topminors}
 Every simple directed graph $H$ has the Erd\H{o}s-P\'osa property for topological minors in tournaments with bounding function $f(k)\in \O_H(k\log k)$. 
\end{theorem}

Observe that compared to the results reported by Raymond in~\cite{R18}, we also give explicit upper bounds on the bounding function that are polynomial in $k$: cubic for immersions and near-linear for topological minors. The presentation of~\cite{R18} does not claim any explicit estimates on the bounding function, as it relies on qualitative results of Chudnovsky et al.~\cite{ChudnovskyFS12} and of Fradkin and Seymour~\cite{FradkinS13}. These results respectively say the following: If a tournament $T$ excludes a fixed digraph $H$ as a immersion (respectively, as a topological minor), then the {\emph{cutwidth}} of $T$ (respetively, {\emph{pathwidth}}) of $T$ is bounded by a constant $c_H$ that depends only on $H$. Instead of relying on the results of~\cite{ChudnovskyFS12,FradkinS13}, we point out that we can use their quantitative improvements of Fomin and the second author~\cite{FoPi}, and thus obtain concrete bounds on the bounding function that are polynomial in $k$.

However, the bulk of our work concerns treating directed graphs $H$ that are possibly not strongly connected. Similarly to Raymond~\cite{R18}, using the results of~\cite{FoPi} we may restrict attention to tournaments of bounded cutwidth or pathwidth, which in both cases provides us with a suitable linear ``layout'' of the tournament. Then we analyze how an immersion or a topological minor copy of $H$ can look in this layout, and in particular how the strongly connected components of $H$ are ordered by it. The main point is to focus on every topological ordering of the strongly connected components of $H$ separately. Namely, we show that for a given topological ordering $\pi$, we can either find $k$ disjoint copies of $H$ respecting this ordering in the layout, or uncover a small hitting set for all copies respecting $\pi$. Then taking the union of the hitting sets for all topological orderings $\pi$ finishes the proof.

We do not expect the estimates on the bounding function given by Theorems~\ref{thm:main-imm} and~\ref{thm:main-topminors} to be optimal.
In fact, on the way to proving Theorem~\ref{thm:main-imm} we establish an improved bound of $\O_H(k^2)$ under the assumption that $H$ is strongly connected, which suggests that the same asymptotic bound (i.e. quadratic instead of cubic) should also hold without this assumption. However, to the best of our knowledge, in both cases it could even be that the optimal bounding function is linear in~$k$. Finding tighter estimates is an interesting open question.

\section{Preliminaries}\label{sec:prelims}

For a positive integer $n$, we denote $[n]\coloneqq \{1,\ldots,n\}$. Throughout this paper, all logarithms are base $2$.

We use standard graph terminology and notation.
All graphs considered in this paper are finite, simple (i.e. without self-loops or multiple arcs with same head and tail), 
 and directed (i.e. are {\emph{digraphs}}). For a digraph $D$, by $V(D)$ and $A(D)$ we denote the vertex set and the arc set of $D$, respectively. We denote
$$|D|\coloneqq |V(D)|\qquad \textrm{and}\qquad \|D\|\coloneqq |V(D)|+|A(D)|.$$
For $X\subseteq V(D)$, the subgraph {\emph{induced}} by $X$, denoted $D[X]$, comprises of the vertices of $X$ and all the arcs of $D$ with both endpoints in $X$.
By $D-X$ we denote the digraph $D[V(D)\setminus X]$. Further, if $F$ is a subset of arcs of $D$, then by $D-F$ we denote the digraph obtained from $D$ by removing all the arcs of $F$.

A {\emph{strong component}} of $D$ is an inclusion-wise maximal induced subgraph $C$ of $D$ that is {\emph{strongly connected}}, that is, for every pair of vertices $u$ and $v$ of $C$, there are directed paths in $C$ both from $u$ to $v$ and from $v$ to $u$.

\pagebreak[3]
\paragraph*{Tournaments.}
A simple digraph $T=(V,A)$ is called a \emph{tournament} if for every pair of distinct vertices $u,v\in V$, either $(u,v)\in A$, or $(v,u)\in A$ (but not both). Alternatively, one can represent the tournament $T$ by providing a pair $(\sigma,\barcs_\sigma(T))$, where $\sigma\colon V\to [|V|]$ is an {\emph{ordering}} of the set $V$ and $\barcs_{\sigma}(T)$ is the set of \emph{$\sigma$-backward arcs}, that is, $$\barcs_{\sigma}(T)\coloneqq\{\,(u,v)\in A\ \mid\ \sigma(u)>\sigma(v)\,\}.$$ All the remaining arcs are called \emph{$\sigma$-forward}. If the choice of ordering $\sigma$ is clear from the context, we will call the arcs simply \emph{backward} or \emph{forward}.

\paragraph*{Cutwidth.}
Let $T=(V,A)$ be a tournament and $\sigma$ be an ordering of $V$. For $\alpha,\beta\in \{0,1,\ldots,|V|\}$, $\alpha\leq \beta$, we define $$\sigma(\alpha,\beta]\coloneqq \{v\in V\mid \alpha<\sigma(v)\leq\beta\}.$$ Sets $\sigma(\alpha,\beta]$ as defined above shall be called \emph{$\sigma$-intervals}. If $I=\sigma(\alpha,\beta]$, we denote 
$$\first_\sigma(I)\coloneqq \alpha\qquad\textrm{and}\qquad \last_\sigma(I)\coloneqq \beta.$$ Moreover, let $\sigma[\alpha]\coloneqq \sigma(0,\alpha]$ and call this interval an \emph{$\alpha$-prefix} of $\sigma$. The set\[\cut[\alpha]=\{(u,v)\in A\mid \sigma(u)>\alpha\geq\sigma(v)\}\subseteq\barcs_\sigma(T)\] is called the \emph{$\alpha$-cut} of $\sigma$. The {\emph{width}} of the ordering $\sigma$ is equal to $\max_{0\leq\alpha\leq |V|}|\cut[\alpha]|$, and the \emph{cutwidth} of~$T$, denoted $\ctw(T)$, is the minimum width among all orderings of $V$. 

\paragraph*{Immersions.}
Digraph $\wh{H}$ is an \emph{immersion model} (or an \emph{immersion copy}) of a digraph $H$ if there exists a mapping $\phi$, called an {\emph{immersion embedding}}, such that: 
\begin{itemize}[nosep]
\item vertices of $H$ are mapped to pairwise different vertices of $\whH$;
\item each arc $(u,v)\in A(H)$ is mapped to a directed path in $\whH$ starting at $\phi(u)$ and ending at $\phi(v)$; and
\item each arc of $\whH$ belongs to exactly one of the paths $\{\phi(a)\colon a\in A(H)\}$.
\end{itemize}
If the immersion embedding $\phi$ is clear for the context, then for a subgraph $C$ of $H$ we define $\whH|_C$ to be the subgraph of $\whH$ consisting of all the vertices and arcs participating in the image of $C$ under $\phi$. Note that thus, $\whH|_C$ is an immersion model of $C$.

Let $H$ be a digraph.
We say that a digraph $G$ \emph{contains $H$ as an immersion} (or $H$ can be \emph{immersed in} $G$) if $G$ has a subgraph that is an immersion model of $H$. Digraph $G$ is called \emph{$H$-immersion-free} if it does not contain $H$ as an immersion.

We will use the following result of Fomin and the second author.

\begin{theorem}[Theorem~7.3 of \cite{FoPi}]\label{thm:ctw-bound}
Let $T$ be a tournament which does not contain a digraph $H$ as an immersion. Then $\ctw(T)\in \O(\|H\|^2)$.
\end{theorem}

From Theorem~\ref{thm:ctw-bound} we can derive the following statement.

\begin{corollary}\label{cor:no-kH}
Let $T$ be a tournament which does not contain $k$ arc-disjoint immersion copies of a digraph $H$. Then $\ctw(T)\in \O(\|H\|^2k^2)$.
\end{corollary}
\begin{proof}
 Let $D$ be the digraph obtained by taking $k$ vertex-disjoint copies of $H$. Clearly, $T$ does not contain $D$ as an immersion, hence from Theorem~\ref{thm:ctw-bound} we conclude that $\ctw(T)\in \O(\|D\|^2)=\O(\|H\|^2k^2)$. 
\end{proof}

\paragraph*{Pathwidth.} Denote by $\I$ the set of all nonempty intervals $[\alpha,\beta]\subseteq \mathbb{R}$ such that $\alpha,\beta\in \mathbb{Z}$. If $I=[\alpha,\beta]$, denote $\first(I)\coloneqq \alpha$ and $\last(I)\coloneqq \beta$. For $I,J\in\I$ we will write $I<J$ if and only if $\last(I)<\first(J)$.

For a tournament $T=(V,A)$, a function $I\colon V\to \I$ is called an \emph{interval decomposition} of $T$ if  for every pair of vertices $u,v\in V$ such that $I(u)<I(v)$, we have $(u,v)\in A$. In other words, every arc joining disjoint intervals is \emph{forward}. For $\alpha\in \mathbb{Z}$, the set
\[\vcut[\alpha]\coloneqq\{v\in V\mid \alpha\in I(v)\}\]
is called the \emph{$\alpha$-cut} of $I$. The \emph{width} of the decomposition $I$ is equal to $\max_{\alpha\in\mathbb{Z}}|\vcut[\alpha]|$, and the \emph{pathwidth} of $T$, denoted $\pw(T)$, is the minimum width among all interval decompositions of $T$.

Let us remark here that the definition of pathwidth used in~\cite{FoPi} is seemingly somewhat different to the one delivered above: it is based on a notion of a {\emph{path decomposition}}, which is a sequence of {\emph{bags}} that correspond to sets $\{\vcut[\alpha]\colon \alpha\in \mathbb{Z}\}$ in an interval decomposition. However, it is straightforward to verify that the definitions are in fact equivalent.

Also, it is easy to see that given an interval decomposition $I$ of a tournament $T$, one can adjust $I$ to an interval decomposition~$I'$ of the same width where no two intervals share an endpoint and no interval has length $0$. Indeed, whenever a subset of intervals all have endpoints at $\alpha\in \mathbb{Z}$, then one can shift those endpoints by pairwise different small reals --- positive for the intervals ending at $\alpha$ and negative for those starting at $\alpha$ --- so that they all become different, and then re-enumerate all the endpoints so that they stay integral. Similarly one can stretch an interval of length $0$ which doesn't share endpoints with any other interval to an interval of positive length.
Therefore, we will assume this property for all the considered interval decompositions:  $\{\first(I(u)),\last(I(u))\}\cap\{\first(I(v)),\last(I(v))\}=\varnothing$ for all $u\neq v$ and $\first(I(u))\neq \last(I(u))$ for all $u$. Moreover, by shifting all the intervals if necessary, we may (and will) assume that all endpoints correspond to non-negative integers.

If $I$ is an interval decomposition of a tournament $T=(V,A)$, then for $\alpha,\beta\in\mathbb{Z}$ we define
\[I[\alpha,\beta]\coloneqq\{v\in V\mid I(v)\subseteq [\alpha,\beta]\}.\]
In other words, $I[\alpha,\beta]$ is the set of all vertices of $T$ corresponding to intervals entirely contained in $[\alpha,\beta]$. Note that if $\alpha_1<\beta_1\leq\alpha_2<\beta_2$, then $I[\alpha_1,\beta_1]\cap I[\alpha_2,\beta_2]=\varnothing$. Also, let $I[\alpha]\coloneqq I[0,\alpha]$.

\paragraph*{Topological minors.} Digraph $\wh{H}$ is a \emph{topological minor model} (or a \emph{topological minor copy}) of a digraph $H$ if there exists a mapping $\phi$, called a {\emph{topological minor embedding}}, such that: 
\begin{itemize}[nosep]
\item vertices of $H$ are mapped to pairwise different vertices of $\whH$;
\item each arc $(u,v)\in A(H)$ is mapped to a directed path in $\whH$ starting at $\phi(u)$ and ending at $\phi(v)$; and
\item these paths are internally vertex-disjoint, do not contain any $\phi(u)$, $u\in V(H)$, as an internal vertex, and saturate the whole vertex set and arc set of $\whH$. In other words, every arc of $\whH$  and every vertex of $\whH$ that is not an image of a vertex of $H$ participates in the image $\phi(a)$ of exactly one arc $a\in A(H)$.
\end{itemize}
If the topological minor embedding $\phi$ is clear for the context, then for a subgraph $C$ of $H$ we define $\whH|_C$ to be the subgraph of $\whH$ consisting of all the vertices and arcs participating in the image of $C$ under $\phi$. Note that thus, $\whH|_C$ is a topological minor model of $C$.

Let $H$ be a digraph.
We say that a digraph $G$ \emph{contains $H$ as a topological minor} if $G$ has a subgraph that is a topological minor model of $H$. Digraph $G$ is called \emph{$H$-topological-minor-free} if it does not contain $H$ as a topological minor.

We will use another result of Fomin and the second author.

\begin{theorem}[Theorem~7.1 of \cite{FoPi}]\label{thm:pw-bound}
Let $T$ be a tournament which does not contain a digraph $H$ as a topological minor. Then $\pw(T)\in \O(\|H\|)$.
\end{theorem}

Applying Theorem~\ref{thm:pw-bound} directly to the graph that is the disjoint union of $k$ copies of a fixed digraph, we can derive the following statement.

\begin{corollary}\label{cor:top-no-kH}
Let $T$ be a tournament that does not contain $k$ vertex-disjoint topological minor copies of a digraph~$H$. Then $\pw(T)\in \O(\|H\|k)$.
\end{corollary}

\newcommand{\dctw}{d_{\mathrm{ctw}}}
\newcommand{\dpw}{d_{\mathrm{pw}}}
\newcommand{\dEH}{d_{\mathrm{eh}}}

\section{Erdős-Pósa property for immersions}

In this section we prove Theorem~\ref{thm:main-imm}. In the following, a subset of arcs $F$ in a digraph $D$ is {\emph{$H$-hitting}} if the digraph $D-F$ is $H$-immersion-free. We also fix the constant $\dctw$ hidden in the $\O(\cdot)$-notation in Theorem~\ref{thm:ctw-bound}; that is, if a tournament $T$ does not contain $H$ as an immersion then $\ctw(T)\leq \dctw\|H\|^2$. Note that the constant hidden in the $\O(\cdot)$-notation in Corollary~\ref{cor:no-kH} is also equal to $\dctw$. Without loss of generality we assume that $\dctw$ is the square of an even integer.

We start with two straightforward observations which will be used several times later on.


\begin{observation}\label{obs:comp-map}
Suppose $\whH$ is an immersion model of a digraph $H$ in a digraph $G$, and $C$ is a strong component of $H$. Then there exists a strong component $D$ of $G$ such that $\whH|_C$ is a subgraph of $D$.
\end{observation}

\begin{observation}\label{obs:strong}
Let $T$ be a tournament and $\sigma$ be an ordering of $V(T)$. Let $H$ be a strongly connected simple digraph with at least one arc and let $\whH$ be an immersion model of $H$ in $T$. Let $v$ be the vertex of $V(\whH)$ that is last in the ordering $\sigma$. Then $A(\whH)$ contains a $\sigma$-backward arc with tail $v$.
\end{observation}

We now consider two special cases: when $H$ is acyclic and when $H$ is strongly connected. 
For the acyclic case, we will use the following corollary of the classic results of Erd\H{o}s and Hanani~\cite{EH}.

\begin{lemma}[follows from~\cite{EH}]\label{lem:EH}
 There exists a universal constant $\dEH$ such that for all positive integers $q,k$, in a complete graph on at least $\dEH\cdot q\sqrt{k}$ vertices one can find $k$ pairwise arc-disjoint complete subgraphs, each on $q$ vertices.   
\end{lemma}

From now on, we adopt the constant $\dEH$ in the notation.

\begin{lemma}\label{lem:hitting-acyclic}
Let $H$ be an acyclic simple digraph and let $T$ be a tournament such that $|T|\geq \dEH\cdot 2^{|H|}\sqrt{k}$. Then $T$ contains $k$ arc-disjoint subgraphs isomorphic to $H$.
\end{lemma}
\begin{proof}
 By Lemma~\ref{lem:EH}, in $T$ one can find $k$ arc-disjoint subtournaments $T_1,\ldots,T_k$, each on $2^{|H|}$ vertices. It is well-known that a tournament on $2^{|H|}$ vertices contains a transitive (i.e. acyclic) subtournament on $|H|$ vertices. As $H$ is acyclic, it is a subgraph of a transitive tournament on $|H|$ vertices. Hence, each of $T_1,\ldots,T_k$ contains a subgraph isomorphic to $H$, and these subgraphs are arc-disjoint. 
\end{proof}

%

\begin{corollary}\label{cor:hitting-acyclic}
Let $H$ be a simple digraph that is acyclic and let $k$ be a positive integer. Let $T$ be a tournament that does not contain $k$ arc-disjoint immersion copies of $H$. Then one can find in $T$ a set of at most $\dEH^2\cdot 4^{|H|} k$ arcs that is $H$-hitting.
\end{corollary}
\begin{proof}
We first consider the corner case when $H$ does not contain any arc. Then $T$ must have less than $|H|$ vertices, for otherwise repeating any set of $|H|$ vertices $k$ times would yield $k$ arc-disjoint immersion copies of $H$. Therefore, $T$ in fact does not contain any immersion copy of $H$, due to having less vertices, and the empty set is $H$-hitting in $T$.

Hence, let us assume that $H$ contains at least one arc.
Observe that $|T|<\dEH\cdot 2^{|H|} \sqrt{k}$, for otherwise, by Lemma~\ref{lem:hitting-acyclic}, there would exist $k$ arc-disjoint immersion copies of $H$ in $T$.
Since $H$ has at least one arc, the set $A(T)$ of all the arcs of $T$ is $H$-hitting, and this set has size at most $\binom{\dEH\cdot 2^{|H|} \sqrt{k}}{2}\leq \dEH^2\cdot 4^{|H|} k$, as requested. 
\end{proof}

We now move to the case when $H$ is strongly connected. Recall that this case was already considered by Raymond~\cite{R18}, but we give a more refined argument that gives precise upper bounds on the bounding function. The proof relies on an strategy of finding a cut that separates the immersion copies of $H$ in a roughly balanced way, and applying induction to each side of the cut. This strategy has been applied before in the context of Erd\H{o}s-P\'osa properties, see e.g.~\cite{FominST11}.


\begin{lemma}\label{lem:hitting-strongly}
Let $H$ be a simple digraph that is strongly connected and contains at least one arc, and let $k$ be a positive integer. Let $T$ be a tournament that does not contain $k$ arc-disjoint immersion copies of $H$. Then one can find in $T$ a set of at most $6\dctw\|H\|^2 k^2$ arcs that is $H$-hitting.
\end{lemma}
\begin{proof}
 We prove the lemma by induction on $k$. For the base case $k=1$, $T$ does not contain any immersion copy of $H$, hence the empty set is $H$-hitting.
 
 Assume then that $k\geq 2$. By Corollary~\ref{cor:no-kH}, there is an ordering $\sigma$ of $V(T)$ of width at most $\dctw \|H\|^2 k^2$. Let $\alpha\in \{0,1,\ldots,|V(T)|\}$ be the largest index such that the tournament $T[\sigma[\alpha]]$ does not contain $\lceil k/2\rceil$ arc-disjoint immersion copies of $H$. Since $\lceil k/2\rceil<k$, by induction there exists a set of arcs $F_1$ of size at most $6\dctw\|H\|^2 \lceil k/2\rceil^2$ that is $H$-hitting in $T[\sigma[\alpha]]$. 
 
 If $\alpha=|V(T)|$, or equivalently $T[\sigma[\alpha]]=T$, then $F_1$ is in fact $H$-hitting in $T$ and we are done. Hence, we assume from now on that $\alpha<|V(T)|$. By the maximality of $\alpha$ we know that $T[\sigma[\alpha+1]]$ contains $\lceil k/2\rceil$ arc-disjoint immersion copies of $H$. It follows that the tournament $T-\sigma[\alpha+1]$ does not contain $\lfloor k/2\rfloor$ arc-disjoint immersion copies of $H$, for otherwise together we would expose $\lceil k/2\rceil+\lfloor k/2\rfloor=k$ arc-disjoint immersion copies of $H$ in $T$. By induction, there exists a set of arcs $F_2$ of size at most $6\dctw\|H\|^2 \lfloor k/2\rfloor^2$ that is $H$-hitting in $T-\sigma[\alpha+1]$.
 
 Let now
 $$F\coloneqq F_1\cup F_2\cup \cut[\alpha]\cup \cut[\alpha+1].$$
 Observe that $F$ is $H$-hitting in $T$. Indeed, since $H$ is strongly connected and has at least one arc, every immersion copy of $H$ in $T$ that in not entirely contained in $T[\sigma[\alpha]]$ or $T-\sigma[\alpha+1]$ is hit by $\cut[\alpha]\cup \cut[\alpha+1]$, whereas immersion copies entirely contained in $T[\sigma[\alpha]]$ and in $T-\sigma[\alpha+1]$ are hit by $F_1$ and $F_2$, respectively. It remains to estimate the size of $F$:
 \begin{eqnarray*}
  |F| & \leq & |F_1|+|F_2|+|\cut[\alpha]|+|\cut[\alpha+1]|\\\
      & \leq & 6\dctw\|H\|^2\left(\lceil k/2\rceil^2+\lfloor k/2\rfloor^2\right)+2\dctw\|H\|^2k^2\\
      & \leq & 6\dctw\|H\|^2\left(\left(\frac{k+1}{2}\right)^2+\left(\frac{k-1}{2}\right)^2\right)+2\dctw\|H\|^2k^2\\
      & =    & \dctw\|H\|^2\left(3(k^2+1)+2k^2\right) \leq 6\dctw\|H\|^2k^2.
 \end{eqnarray*}
 This concludes the inductive proof.
\end{proof}

Actually, in our later proof we will not be able to rely on Lemma~\ref{lem:hitting-strongly} for the following reason: we will need the copies to be vertex-disjoint, rather than arc-disjoint. The following statement is tailored to vertex-disjointness. 

\begin{lemma}\label{lem:hitting-strongly-vrtx}
Let $H$ be a simple digraph that is  strongly connected and contains at least one arc, and let $k$ be a positive integer. Further, let $T$ be a tournament with $\ctw(T)\leq c$ that does not contain $k$ vertex-disjoint immersion copies of $H$. Then one can find in $T$  a set of at most $2(k-1)c$ arcs that is $H$-hitting.
\end{lemma}
\begin{proof}
We proceed by induction on $k$. Let $\sigma$ be an ordering of $T$ of width at most $c$. If $T$ does not contain any copy of $H$, then the empty set is $H$-hitting. This proves the base case $k=1$, so from now on we may assume that $k\geq 2$ and that $T$ contains at least one immersion copy of $H$.

Let $\alpha$ be the minimum integer satisfying the following: $T[\sigma[\alpha]]$ contains an immersion copy $\whH$ of~$H$. Let $B_1\coloneqq\{(u,v)\in\barcs_\sigma(T)\mid \sigma(u)=\alpha\}$ be the set of backward arcs with tail $\alpha$ 
and let $B\coloneqq B_1\cup\cut[\alpha]$. As $B_1\subseteq \cut[\alpha-1]$, we have $|B|\leq 2c$.

Observe that in $T'\coloneqq T[V(T)\setminus\sigma[\alpha]]$ one cannot find a family of $k-1$ vertex-disjoint immersion copies of $H$. Indeed, if there was such a family, then adding $\whH$ to it would yield a family of $k$ vertex-disjoint copies of $H$ in $T$, a contradiction. Hence, by induction hypothesis, in $T'$ there is a set $S$ of at most $2(k-2)c$ arcs that is $H$-hitting. We claim that the set $B\cup S$ is $H$-hitting in $T$. Note that since $|B\cup S|\leq |B|+|S|\leq 2(k-1)c$, this will conclude the proof.

Indeed, suppose that $\whH'$ is an immersion copy of $H$ in $T-(B\cup S)$. By Observation~\ref{obs:comp-map}, either $V(\whH')\subseteq V\setminus\sigma[\alpha]$, or $V(\whH')\subseteq \sigma[\alpha]$. The first case is impossible, because every immersion copy of $H$ in $T'$ contains an arc from $S$. On the other hand, if $V(\whH')\subseteq \sigma[\alpha]$, then by the minimality of $\alpha$ we infer that $\sigma^{-1}(\alpha)\in V(\whH')$. Then Observation~\ref{obs:strong} implies that $\whH'$ needs to contain an arc of $B_1$, again a contradiction.
\end{proof}

Note that by combining Lemma~\ref{lem:hitting-strongly-vrtx} with Corollary~\ref{cor:no-kH}, we obtain a statement analogous to Lemma~\ref{lem:hitting-strongly}, however with a bound of $\O(k^3)$ instead of $\O(k^2)$. This drawback will accordingly affect the final dependency on $k$ in Theorem~\ref{thm:main-imm}.

We now proceed to the main part of the proof, which concerns digraphs that are not acyclic and that are not necessarily strongly connected. 

\def\bb{b}
\def\dd{d}

\begin{lemma}\label{lem:hitting-weakly}
Let $H$ be a simple digraph that is not acyclic and let $k$ be a positive integer. Let $T$ be a tournament that does not contain $k$ arc-disjoint immersion copies of $H$. Then one can find in $T$ a set consisting of at most $2\dctw^{3/2}\cdot|H|!\cdot |H|\cdot \|H\|^3\cdot k^3$ arcs that is $H$-hitting.
\end{lemma}
\begin{proof}
Let $\Comps$ be the family of all strong components of $H$ and let $h\coloneqq |\Comps|$. Since $H$ is not acyclic, $\Comps$ contains at least one strong component $C$ that is {\emph{non-trivial}}, that is, $|C|>1$. In particular, $h\leq |H|-1$. 
Further, let $\Pi$ be the set of all {\emph{topological orderings}} of the strong components of $H$; that is, the elements of $\Pi$ are orderings $\pi\colon \Comps\to [|\Comps|]$ such that for every arc of $H$ with tail in $C\in \Comps$ and head in $D\in \Comps$, we have $\pi(C)\leq \pi(D)$. It is well-known that $\Pi\neq \varnothing$. Also, note that $|\Pi|\leq h!\leq (|H|-1)!$.

Let $T=(V,E)$. By Corollary~\ref{cor:no-kH}, there is an ordering $\sigma$ of vertices of $T$ of width at most $c$, where
$$c\coloneqq \dctw\|H\|^2k^2.$$
We also define
$$s\coloneqq \sqrt{\dctw}\cdot h\|H\|k.$$
Note that thus, $s=h\sqrt{c}$ and $s$ is an even integer, because we assume $\dctw$ to be a square of an even integer.

Let $\I$ be the set of all $\sigma$-intervals. For $I\in\I$, we define 
$$\cut^-(I)\coloneqq\cut[\first_\sigma(I)]\qquad\textrm{and}\qquad\cut^+(I)\coloneqq \cut[\last_\sigma(I)].$$
We define functions 
$$I\colon \Comps\times[|V|]\to \I\qquad\textrm{and}\qquad A,B\colon \Comps\times[|V|]\to \Pow\left(\barcs_\sigma(T)\right),$$ where $\Pow(X)$ denotes the power set of $X$, as follows:
\begin{itemize}[nosep]
\item $I(C,\alpha)$ is the inclusion-wise minimal $\sigma$-interval $I$ such that $\first_\sigma(I)=\alpha$ and $T[I]$ contains at least $s$ vertex-disjoint immersion copies of $C$. If no such interval exists, we set $I(C,\alpha)\coloneqq \sigma(\alpha,|V|]$. Note that either way, $T[I]$ does not contain $s+1$ vertex-disjoint immersion copies of $C$.
\item If $C$ is trivial, then $A(C,\alpha)$ is the set of all backward arcs contained in $T[I(C,\alpha)]$. If $C$ is non-trivial, then $A(C,\alpha)$ is a set of arcs that is $C$-hitting in $T[I(C,\alpha)]$ and is of size at most $2sc$, whose existence follows from Lemma~\ref{lem:hitting-strongly-vrtx}.
\item $B(C,\alpha)\coloneqq \cut^+(I(C,\alpha))$.
\end{itemize}
Note that if $C$ is trivial, then $|I(C,\alpha)|\leq s$. This implies that $|A(C,\alpha)|\leq \binom{s}{2}\leq 2sc$. Hence, in all cases we have 
$$|A(C,\alpha)|\leq 2sc\qquad\textrm{and}\qquad |B(C,\alpha)|\leq c.$$


Consider an arbitrary topological ordering $\pi\in\Pi$. We define indices $\alpha_0,\alpha_1,\ldots,\alpha_h$ and intervals $I_{\pi,1},I_{\pi,2},\ldots,I_{\pi,h}$ by induction as follows: $\alpha_0\coloneqq 0$ and, for $i=1,2,\ldots,h$, we set
\[I_{\pi,i}\coloneqq I(\pi^{-1}(i),\alpha_{i-1})\qquad\textrm{and}\qquad \alpha_i\coloneqq \last_\sigma(I_{\pi,i}),\]
where if needed we put $\last_\sigma(\varnothing)=|V|$. Moreover, for $i\in [h]$ we define 
$$A_{\pi,i}\coloneqq A(\pi^{-1}(i),\alpha_{i-1})\qquad\textrm{and}\qquad B_{\pi,i}\coloneqq B(\pi^{-1}(i),\alpha_{i-1}).$$

Our next step is to show that if for some $\pi\in \Pi$, each interval $I_{\pi,i}$ contains $s$ vertex-disjoint immersion copies of $H$, then we get a contradiction: there are $k$ vertex-disjoint immersion copies of $H$ in $T$.
For this, we will use the following auxiliary statement.

\begin{claim}\label{clm:technical-new}
Let $G$ be a graph with vertex set partitioned into disjoint sets $V_1,\ldots,V_h$, each of size $s$. Suppose that for each pair of indices $1\leq i<j\leq h$, there are at most $\frac{s^2}{h^2}$ edges with one endpoint in $V_i$ and second in $V_j$. Then one can find $s/2$ pairwise disjoint independent sets $I_1,\ldots,I_{s/2}$ in $G$ such that each independent set $I_t$, $t\in [s/2]$, contains exactly one vertex from each set $V_i$, $i\in [h]$.
\end{claim}
\begin{clproof}
 For each $i\in [h]$ let us arbitrarily enumerate the vertices of $V_i$ as $v_{i}[0],\ldots,v_{i}[s-1]$. Consider the following random experiment: draw independently and uniformly at random numbers $t,a_1,\ldots,a_h$ from $\{0,1,\ldots,s-1\}$, and let
 $$I\coloneqq \{v_i[(t+a_i)\bmod s]\colon i\in [h]\}.$$
 Note that for each fixed pair of indices $1\leq i<j\leq h$, the probability that there is an edge between vertices $v_i[(t+a_i)\bmod s]$ and $v_j[(t+a_j)\bmod s]$ is bounded by $\frac{1}{h^2}$. By the union bound we infer that $I$ is an independent set with probability at least $\frac{1}{2}$. Hence, there is a choice of $\hat{a}_1,\ldots,\hat{a}_h\in \{0,1,\ldots,s-1\}$ such that conditioned on $a_1=\hat{a}_1,\ldots,a_h=\hat{a}_h$, the probability (over the choice of $t$) that $I$ is an independent set is at least $\frac{1}{2}$. In other words, for at least $s/2$ choices of $t$, the set $\{v_i[(t+\hat{a}_i)\bmod s]\colon i\in [h]\}$ is independent. This gives us the desired family of $s/2$ pairwise disjoint independent sets.
\end{clproof}

\begin{claim}\label{clm:packing}
Suppose that there exists $\pi\in\Pi$ such that for every $i\in [h]$, the tournament $T[I_{\pi,i}]$ contains $s$ vertex-disjoint immersion copies of $\pi^{-1}(i)$. Then $T$ contains $k$ vertex-disjoint immersion copies of $H$.
\end{claim}
\begin{clproof}
Denote $C_{\pi,i}\coloneqq \pi^{-1}(i)$. For each $i\in[h]$, let $\C_i$ be the family of $s$ vertex-disjoint immersion copies of $C_{\pi,i}$ contained in $T[I_{\pi,i}]$.

Let $G$ be a graph on vertex set $\C_1\cup \ldots \cup\C_p$ where for each pair of indices $1\leq i<j\leq h$ and pair of immersion copies $Q\in \C_i$ and $R\in \C_j$, we put an edge if and only if in $T$ there is arc with tail in $R$ and head in $Q$. Note that such an arc is backward in $\sigma$ and belongs to $\cut[\alpha_i]$. Hence, for every pair of indices $i,j$  as above, $G$ contains at most $c$ edges with one endpoint in $\C_i$ and second in $\C_j$. 

Noting that $c=\frac{s^2}{h^2}$, we may apply Claim~\ref{clm:technical-new} to conclude that $G$ contains $s/2$ pairwise independent sets, each consisting of one element from each of the families $\C_1,\ldots,\C_h$. As $s/2\geq k$, let $I_1,\ldots,I_k$ be any $k$ of those independent sets. Now, for each $t\in [k]$, we may construct an immersion copy of $H$ contained in $T[\bigcup_{Q\in I_t} V(Q)]$ as follows: take the union of subgraphs $Q\in I_t$, which are immersion copies of $C_{\pi,1},\ldots,C_{\pi,t}$, respectively, and for each arc $(a,b)$ of $H$ that is not contained in any of $C_{\pi,1},\ldots,C_{\pi,t}$, say $a\in V(C_{\pi,i})$ and $b\in V(C_{\pi,j})$ where we necessarily have $i<j$, map $(a,b)$ to the single edge between the corresponding two vertices from the copies of $C_{\pi,i}$ and $C_{\pi,j}$ in $I_t$. Note that this edge is oriented forward in $\sigma$, because $I_t$ is an independent set in $G$ (a backward arc would have generated an edge in $G$). Thus, we have constructed $k$ vertex-disjoint copies of $H$ in $T$.
\end{clproof}

If the assumption of Claim~\ref{clm:packing} holds, then we immediately obtain a contradicion and the proof is finished. Therefore, we may further assume that for every $\pi\in\Pi(H)$ there exists $j\in [h]$ such that $I_{\pi,j}$ contains less than $s$ vertex-disjoint copies of $\pi^{-1}(j)$. Observe that this implies that $\last_\sigma(I_{\pi,j})=|V|$, hence in particular we have
$$\bigcup_{i=1}^h I_{\pi,i}=V(T)\qquad\textrm{for each }\pi\in \Pi.$$
Let 
\[S\coloneqq \bigcup_{\pi\in\Pi}\,\bigcup_{i=1}^h\, A_{\pi,i}\cup B_{\pi,i}.\]
Observe that
\begin{eqnarray*}
|S|& \leq & |\Pi|\cdot |H|\cdot (2sc+c)\\
& \leq & (|H|-1)!\cdot |H|\cdot (2s+1)c\\
& = & |H|!\cdot (2\sqrt{\dctw}\cdot h\|H\|k+1)\cdot \dctw\|H\|^2k^2\\
& \leq &  2\dctw^{3/2}\cdot|H|!\cdot |H|\cdot \|H\|^3\cdot k^3,
\end{eqnarray*}
so to finish the proof it suffices to show that $S$ is $H$-hitting in $T$. Let $T'\coloneqq T-S$. 

\begin{figure}[ht]
\begin{center}
\includegraphics[scale=1.2]{picture-1.mps}
\end{center}
\caption{Objects defined in the proof of Claim~\ref{clm:hitting} with $h=10$, $m=4$, $n=3$.}
\label{figura}
\end{figure}

\begin{claim}\label{clm:hitting}
$T'$ is $H$-immersion-free.
\end{claim}
\begin{clproof}
Let $\Bb$ be the family of all inclusion-wise maximal $\sigma$-intervals $B$ satisfying the following property: for every $\pi\in\Pi$ and $i\in [h]$, either $B\subseteq I_{\pi,i}$ or $B\cap I_{\pi,i}=\varnothing$. Call elements of $\Bb$ \emph{base intervals} and observe that $\Bb$ is a partition of $V(T')$. Let $\Ff$ be the family of all $\sigma$-intervals which are disjoint unions of collections of base intervals.  For two disjoint intervals $J,J'\in \Ff$, we write $J<J'$ if $\last_\sigma(J)\leq\first_\sigma(J')$.

Suppose for contradiction that $T'$ contains an immersion model $\whH$ of $H$. We fix some immersion embedding of $H$ in $\whH$, to which we will implicitly refer when considering subgraphs $\whH|_C$ for $C\in \Comps$.

Note that in $T'$ there are no backward arcs with endpoints in different intervals from $\Bb$, as $B_{\pi,i}\subseteq S$ for every $\pi\in \Pi$ and $i\in [h]$. Hence, every non-trivial strongly connected subgraph of $T'$ must have all vertices contained in a single base interval. In particular, from Observation~\ref{obs:comp-map} we infer that for every non-trivial strong component $C\in \Comps$, the subgraph $\whH|_C$ has all its vertices contained in a single base interval. Note that this conclusion also holds trivially when $C$ is trivial.

Let $B_1<B_2<\ldots<B_m$ be all the base intervals containing subgraphs $\whH|_C$ for non-trivial components $C\in \Comps$. Consider any partition of $V(T)$ into intervals $J_1,J_2,\ldots,J_m\in \Ff$ such that $B_i\subseteq J_i$ for each $i\in [m]$. Note that this implies that $J_1<J_2<\ldots<J_m$.
For each $i\in [m]$, let $\Comps_i\subseteq \Comps$ be the set of all (including trivial) components $C\in \Comps$ such that $V(\whH|_C)\subseteq J_i$. Note that $\{\Comps_i\colon i\in [m]\}$ is a partition of $\Comps$ and each family $\Comps_i$ contains at least one non-trivial component.

Observe that if $C\in \Comps_i$ and $C'\in \Comps_{i'}$, where $i\neq i'$, and in $H$ there is an arc $(u,v)$ with $u\in V(C)$ and $v\in V(C')$, then we necessarily have $i<i'$. Indeed, the image of $(u,v)$ in the immersion embedding is a path in $T'$ that starts in $J_i$ and ends in $J_{i'}$, while in $T'$ arcs with endpoints in different intervals among $\{J_1,\ldots,J_m\}$ always point from an interval with a smaller index to an interval with a higher index.
Therefore, there exists a topological ordering $\pi\in \Pi$ such that for all $i,i'\in [m]$ satisfying $i<i'$, all the components of $\Comps_i$ appear in $\pi$ before all the components of $\Comps_{i'}$. In other words, there exist integers $0=t_0<t_1<t_2<\ldots<t_m=h$ such that for each $i\in [m]$, we have $\pi(\Comps_i)=(t_{i-1},t_i]\cap \mathbb{Z}$ (cf. Figure~\ref{figura}). For every $i\in [m]$, we define
$$L_{i}\coloneqq \bigcup_{C\in\Comps_i}I_{\pi,\pi(C)}.$$
Note that $L_i$ is a $\sigma$-interval belonging to $\Ff$, because the set $\Comps_i$ is contiguous in the ordering $\pi$. Furthermore $\{L_i\colon i\in [m]\}$ is a partition of $V(T)$ and $L_1<L_2<\ldots<L_m$.

Recalling that both $\{J_i\colon i\in [m]\}$ and $\{L_i\colon i\in [m]\}$ are partitions of $V(T)$, we can define $n$ to be the smallest positive integer satisfying $\bigcup_{i=1}^n J_i\subseteq \bigcup_{i=1}^n L_i$. By the minimality of $n$, we have $J_n\subseteq L_n$.

Recall that $B_n\subseteq J_n\subseteq L_n$ and $B_n$ is a base interval. Therefore, there exists $C\in\Comps_n$ such that $B_n\subseteq I_{\pi,\pi(C)}$. If $C$ is non-trivial, then the set of arcs $A_{\pi,\pi(C)}$ is $C$-hitting in $T[I_{\pi,\pi(C)}]$. This implies that $T'[I_{\pi,\pi(C)}]$ is $C$-immersion-free, and so is its subgraph $T'[B_n]$. However, $B_n$ is the only interval among $\{B_1,\ldots,B_m\}$ that is contained in $I_{\pi,\pi(C)}$, hence $C$ being a non-trivial component from $\Comps_n$ implies that $V(\whH|_C)\subseteq B_n$; a contradiction. If $C$ is trivial, then $T[I_{\pi,\pi(C)}]-A_{\pi,\pi(C)}$ is acyclic, hence $T'[I_{\pi,\pi(C)}]$ is $C'$-immersion-free for every non-trivial component $C'\in \Comps_n$. Since there exists such a non-trivial component $C'$ and it again satisfies $V(\whH_{C'})\subseteq B_n\subseteq I_{\pi,\pi(C)}$, we again obtain a contradiction.
\end{clproof}

As argued, Claim~\ref{clm:hitting} finishes the proof of Lemma~\ref{lem:hitting-weakly}.
\end{proof}

With Lemma~\ref{lem:hitting-weakly} in place, we can finish the proof of Theorem~\ref{thm:main-imm}.

\begin{proof}[of Theorem~\ref{thm:main-imm}]
 If $H$ has no arcs, then the statement holds trivially for bounding function $f(k)=0$. 
 Hence, from now on assume that $H$ has at least one arc.
Suppose $T$ is a tournament that does not contain $k$ arc-disjoint immersion copies of $H$. If $H$ is acyclic, then, by Corollary~\ref{cor:hitting-acyclic}, we may find in $T$ a set of at most $\dEH^2\cdot 4^{|H|} k\in\O_H(k)$ arcs that is $H$-hitting. On the other hand, if $H$ is not acyclic, then by Lemma~\ref{lem:hitting-weakly} we may find in $T$ an $H$-hitting set of arcs of size at most $2\dctw^{3/2}\cdot|H|!\cdot |H|\cdot \|H\|^3\cdot k^3\in\O_H(k^3)$.
\end{proof}

\section{Erdős-Pósa property for topological minors}

In this section we prove Theorem~\ref{thm:main-topminors}. The proof follows similar ideas to the ones presented in the previous section, only adjusted to the setting of interval decompositions. Throughout this section, 
the notions of a \emph{copy} and of \emph{hitting} will refer to topological minor~copies. Let us fix the constant $\dpw$ hidden in the $\O(\cdot)$-notation in Theorem~\ref{thm:pw-bound} and note that the constant hidden in the $\O(\cdot)$-notation in Corollary~\ref{cor:top-no-kH} is also equal to $\dpw$.

Consider first the acyclic case. The following statements are analogues of Lemma~\ref{lem:hitting-acyclic} and Corollary~\ref{cor:hitting-acyclic}.

\begin{lemma}\label{lem:top-hitting-acyclic}
 Let $H$ be an acyclic simple digraph let $T$ be a tournament such that $|T|\geq 2^{|H|} k$. Then $T$ contains $k$ vertex-disjoint subgraphs isomorphic to $H$.
\end{lemma}
\begin{proof}
Arbitrarily partition the vertex set of $T$ into subsets $W_1,\ldots,W_k$ so that $|W_i|\geq 2^{|H|}$ for each $i\in [k]$. Since a tournament on $2^{|H|}$ vertices contains a transitive subtournament on $|H|$ vertices, which in turn contains $H$ as a subgraph, we infer that each tournament $T[W_i]$, $i\in [k]$, contains a topological minor copy of $H$. This gives $k$ vertex-disjoint topological minor copies of $H$ in $T$.
\end{proof}

\begin{corollary}\label{cor:top-hitting-acyclic}
Let $H$ be a simple digraph that is acyclic and let $k$ be a positive integer. Let $T$ be a tournament that does not contain $k$ vertex-disjoint topological minor copies of $H$. Then one can find in $T$ a set of at most $2^{|H|} k$ vertices that is $H$-hitting.
\end{corollary}
\begin{proof}
By Lemma~\ref{lem:top-hitting-acyclic} we have $|T|<2^{|H|} k$, so we can take the whole vertex set of $T$ as the requested $H$-hitting set.
\end{proof}

We now proceed to the strongly connected case and prove an analogue of Lemma~\ref{lem:hitting-strongly}. Note that in this setting, we can use the strategy from the proof of Lemma~\ref{lem:hitting-strongly} and directly achieve vertex-disjointness. Hence, we will need no counterpart of Lemma~\ref{lem:hitting-strongly-vrtx}.

\begin{lemma}\label{lem:top-hitting-strongly}
Let $H$ be a strongly connected simple digraph and let $T$ be a tournament that does not contain $k$ vertex-disjoint topological minor copies of $H$. Then in $T$ one can find a set of at most $2\dpw\|H\|\cdot k\log k$ vertices that is $H$-hitting.
\end{lemma}
\begin{proof}
 We proceed by induction on $k$. In the base case $k=1$ there are no  copies of $H$ in $T$, hence we can take the empty set as an $H$-hitting set. Let us then assume that $k\geq 2$.
 
 If $T$ does not contain $\lceil k/2\rceil$ vertex-disjoint copies of $H$, then as $\lceil k/2\rceil<k$, we may apply the induction assumption for $\lceil k/2\rceil$. Hence, from now on assume that $T$ contains $\lceil k/2\rceil$ vertex-disjoint copies of $T$.
 
 By Corollary~\ref{cor:top-no-kH}, $T$ admits an interval decomposition of width at most $\dpw\|H\|\cdot k$. Recall that we may assume that the endpoints of intervals in $I$ correspond to pairwise different nonnegative integers.
 Let $\alpha$ be the largest integer such that $T[I[\alpha]]$ does not contain $\lceil k/2 \rceil$ vertex-disjoint copies of $H$. By the assumption from the previous paragraph, $\alpha$ is well defined and $T[\alpha+1]$ contains $\lceil k/2 \rceil$ vertex-disjoint copies of $H$. It follows that $T[I[\alpha+2,\infty]]$ does not contain $\lfloor k/2 \rfloor$ vertex-disjoint copies of $H$, for otherwise in total we would obtain $\lceil k/2 \rceil+\lfloor k/2 \rfloor=k$ vertex-disjoint copies of $H$.
 
 By induction assumption, in $T[I[\alpha]]$ and in $T[I[\alpha+2,\infty]]$ we can find $H$-hitting sets $S_1$ and $S_2$ of sizes $2\dpw\|H\|\cdot \lceil k/2\rceil \log \lceil k/2\rceil$ and $2\dpw\|H\|\cdot \lfloor k/2\rfloor \log \lfloor k/2\rfloor$, respectively. Let
 $$S\coloneqq S_1\cup S_2\cup \vcut[\alpha+1].$$
 We claim that $S$ is $H$-hitting in $T$. Indeed, since $H$ is strongly connected, every copy of $H$ in $T$ that does not intersect $\vcut[\alpha]$ must be entirely contained either in $T[I[\alpha]]$ or in $T[I[\alpha+2,\infty]]$, and then it intersects $S_1$ or $S_2$, respectively.
 
 We are left with bounding the size of $S$. Observe that
 \begin{eqnarray*}
  |S| & \leq & |S_1|+|S_2|+|\vcut[\alpha+1]|\\\
      & \leq & 2\dpw\|H\|\left(\lceil k/2\rceil\log \lceil k/2\rceil+\lfloor k/2\rfloor\log \lfloor k/2\rfloor\right)+\dpw\|H\|k\\
      & \leq & 2\dpw\|H\|\left(k/2\left(\log\lceil k/2\rceil+\log\lfloor k/2\rfloor\right)+1/2\left(\log\lceil k/2\rceil-\log\lfloor k/2\rfloor\right)\right)+\dpw\|H\|k\\
      & \leq & 2\dpw\|H\|\left(k\log (k/2)+1/2\right)+\dpw\|H\|k\\
      & =    & \dpw\|H\|\left(2k\log k -2k+1+k\right) \leq 2\dpw\|H\|\cdot k\log k.
 \end{eqnarray*}
 This concludes the proof.
\end{proof}

%
%
%
%
%

We proceed to the main part of the proof, which is is again conceptually very close to the one presented in the case of immersions. It is arguably simpler, as we work only with vertex-disjointness.

\begin{lemma}\label{lem:top-hitting-weakly}
Let $H$ be a simple digraph that is not acyclic and let $k$ be a positive integer. Let $T$ be a tournament that does not contain $k$ vertex-disjoint topological minor copies of $H$. Then one can find in $T$ a set consisting of at most $6\dpw \cdot |H|!\cdot \|H\|\cdot k\log k$ vertices that is $H$-hitting.
\end{lemma}
\begin{proof}
Denote $T=(V,E)$. Define $\Comps$, $h$, $\Pi$, topological ordering and (non-)trivial components as in the proof of Lemma~\ref{lem:hitting-weakly}. Note that the same assertions about these objects apply.

By Corollary~\ref{cor:top-no-kH}, $T$ admits an interval decomposition $I$ of width at most 
$$p\coloneqq \dpw\|H\|\cdot k.$$
Recall that we may assume that the endpoints of the intervals of $I$ are pairwise different nonnegative integers.
Let $N$ be the largest interval end, i.e. $N\coloneqq \max\{\last(I(v))\colon v\in V\}$. Define functions
\[\beta\colon \Comps\times\mathbb{Z}\to \mathbb{Z}\qquad\text{and}\qquad A,\ B\colon\Comps\times \mathbb{Z}\to \Pow(V)\]
as follows:
\begin{itemize}[nosep]
\item $\beta(C,\alpha)$ is the minimum integer $\beta$ with the property that interval $T[I[\alpha,\beta]]$ contains at least $k$ vertex-disjoint topological minor copies of $C$. If no such $\beta$ exists, we set $\beta=N$. Note that since we assume that the endpoints of the intervals in $I$ are pairwise different, in either case $T[I[\alpha,\beta]]$ does not contain $k+1$ vertex-disjoint topological minor copies of $C$.
\item If $C$ is trivial, then $A(C,\alpha)=I[\alpha,\beta(C,\alpha)]$. If $C$ is non-trivial, then $A(C,\alpha)$ is a $C$-hitting set of vertices in $T[I[\alpha,\beta(C,\alpha)]]$ of size at most $2\dpw\|H\|\cdot (k+1)\log (k+1)\leq 5\dpw\|H\|\cdot k\log k$, whose existence follows from Lemma~\ref{lem:top-hitting-strongly}.
\item $B(C,\alpha)\coloneqq\vcut[\beta(C,\alpha)]$.
\end{itemize}\pagebreak[3]
Note that since $p\leq 5\dpw\|H\|\cdot k\log k$, for all $C$ and $\alpha$ we have 
$$|A(C,\alpha)|\leq 5\dpw\|H\|\cdot k\log k\qquad\textrm{and}\qquad|B(C,\alpha)|\leq p.$$

Consider an arbitrary $\pi\in \Pi$ and define indices $\alpha_{\pi,0}$, $\alpha_{\pi,1}$, $\ldots$, $\alpha_{\pi,h}$ by induction as follows: $\alpha_{\pi,0}\coloneqq 0$ and, for $i=1,2,\ldots,h$, set
\[\alpha_{\pi,i}\coloneqq \beta(\pi^{-1}(i),\alpha_{\pi,i-1}).\]   
Moreover, for $i\in [h]$ we define \[I_{\pi,i}\coloneqq I[\alpha_{\pi,i-1},\alpha_{\pi,i}],\qquad A_{\pi,i}\coloneqq A(\pi^{-1}(i),\alpha_{\pi,i-1})\qquad\text{and}\qquad B_{\pi,i}\coloneqq B(\pi^{-1}(i),\alpha_{\pi,i-1}).\]
Note that since no interval in the decomposition $I$ has length $0$, sets $I_{\pi,i}$ for $i\in [h]$ are pairwise disjoint. Moreover, since intervals in $I$ have pairwise different endpoints, for all $1\leq i<j\leq h$, all arcs with one endpoint in $I_{\pi,i}$ and second in $I_{\pi,j}$ have tail in $I_{\pi,i}$ and head in $I_{\pi,j}$.

The following statement can be proved using the same arguments as the corresponding claim  in the proof of Lemma~\ref{lem:hitting-weakly} (that is, Claim~\ref{clm:packing}).
 We simply join the copies of strong components of $H$ by single forward arcs between intervals $I_{\pi,i}$. 

\begin{claim}\label{clm:top-packing}
Suppose that there exists $\pi\in\Pi$ such that for every $i\in [h]$, the tournament $T[I_{\pi,i}]$ contains $k$ vertex-disjoint topological minor copies of $\pi^{-1}(i)$. Then $T$ contains $k$ vertex-disjoint topological minor copies of $H$. 
\end{claim}

Just as in the proof of Lemma~\ref{lem:hitting-weakly}, due to Claim~\ref{clm:top-packing} we may now assume that
\[\bigcup_{i=1}^{h} I_{\pi,i}=V(T)\qquad\text{for each $\pi\in \Pi$.}\]
Consider 
\[S\coloneqq \bigcup_{\pi\in\Pi}\bigcup_{i\in [h]} A_{\pi,i}\cup B_{\pi,i}.\]  Since $|\Pi|\leq (|H|-1)!$ due to $H$ not being acyclic, we have
$$|S|\leq (|H|-1)!\cdot |H|\cdot (5\dpw\|H\|\cdot k\log k+\dpw\|H\|\cdot k)\leq 6\dpw \cdot |H|!\cdot \|H\|\cdot k\log k.$$
So it is enough to prove that $S$ is $H$-hitting in $T$. Let $T'\coloneqq T-S$.

\begin{claim}\label{clm:top-hitting}
$T'$ is $H$-topological-minor-free.
\end{claim}
\begin{clproof} 
The proof follows precisely the same steps as the one of Claim~\ref{clm:hitting}, with minor and straightforward adjustments (e.g. instead of $\sigma$-intervals we consider simply intervals).
\end{clproof}

Claim~\ref{clm:top-hitting} finishes the proof of Lemma~\ref{lem:top-hitting-weakly}.
\end{proof}

Now we can finish the proof of Theorem~\ref{thm:main-topminors}.

\begin{proof}[of Theorem~\ref{thm:main-topminors}]
Suppose $T$ is a tournament that does not contain $k$ vertex-disjoint topological minor copies of $H$. If $H$ is acyclic, then, by Corollary~\ref{cor:top-hitting-acyclic}, we may find in $T$ a set of at most $2^{|H|}k\in \O_H(k)$ vertices that is $H$-hitting. On the other hand, if $H$ is not acyclic, then by Lemma~\ref{lem:top-hitting-weakly} we may find in $T$ an $H$-hitting set of vertices of size at most $6\dpw \cdot |H|!\cdot \|H\|\cdot k\log k\in\O_H(k\log k)$.
\end{proof}

\acknowledgements
The authors thank Jean-Florent Raymond for (i) suggesting the inductive strategy used in the proofs of Lemmas~\ref{lem:hitting-strongly} and~\ref{lem:top-hitting-strongly}, which in particular resulted in improving the bounding function from $\O_H(k^3)$ to $\O_H(k^2)$ for the Erd\H{o}s-P\'osa property for immersions of a strongly connected $H$, and from $\O_H(k^2)$ to $\O_H(k\log k)$ for the Erd\H{o}s-P\'osa property for topological minors, (ii) drawing our attention to the work of Fomin et al.~\cite{FominST11}, and (iii) many other comments that helped us in improving this manuscript.

\nocite{*}
\bibliographystyle{abbrv}
\bibliography{main-dmtcs}

\begin{thebibliography}{10}

\bibitem{AmiriKKW16}
S.~A. Amiri, K.~Kawarabayashi, S.~Kreutzer, and P.~Wollan.
\newblock The {E}rd{\H{o}}s-{P}{\'{o}}sa property for directed graphs.
\newblock {\em CoRR}, abs/1603.02504, 2016.

\bibitem{Be19}
S.~Bessy, M.~Bougeret, R.~Krithika, A.~Sahu, S.~Saurabh, J.~Thiebaut, and
  M.~Zehavi.
\newblock Packing arc-disjoint cycles in tournaments.
\newblock In {\em Proceedings of the 44$^{\textrm{th}}$ International Symposium
  on Mathematical Foundations of Computer Science, {MFCS} 2019}, volume 138 of
  {\em LIPIcs}, pages 27:1--27:14. Schloss Dagstuhl --- Leibniz-Zentrum
  f{\"{u}}r Informatik, 2019.

\bibitem{ChudnovskyFS12}
M.~Chudnovsky, A.~Fradkin, and P.~D. Seymour.
\newblock Tournament immersion and cutwidth.
\newblock {\em Journal of Combinatorial Theory, Series {B}}, 102(1):93--101,
  2012.

\bibitem{EH}
P.~Erd\H{o}s and H.~Hanani.
\newblock On a limit theorem in combinatorial analysis.
\newblock {\em Publ. Math. Debrecen}, 10:10--13, 1963.

\bibitem{EP}
P.~Erd\H{o}s and L.~P\'osa.
\newblock On independent circuits contained in a graph.
\newblock {\em Canadian Journal of Mathematics}, 17:347--352, 1965.

\bibitem{FoPi}
F.~V. Fomin and M.~Pilipczuk.
\newblock On width measures and topological problems on semi-complete digraphs.
\newblock {\em Journal of Combinatorial Theory, Series {B}}, 138:78--165, 2019.

\bibitem{FominST11}
F.~V. Fomin, S.~Saurabh, and D.~M. Thilikos.
\newblock Strengthening {E}rd{\H{o}}s-{P}{\'{o}}sa property for minor-closed
  graph classes.
\newblock {\em J. Graph Theory}, 66(3):235--240, 2011.

\bibitem{FradkinS13}
A.~Fradkin and P.~D. Seymour.
\newblock Tournament pathwidth and topological containment.
\newblock {\em Journal of Combinatorial Theory, Series {B}}, 103(3):374--384,
  2013.

\bibitem{jfr}
J.~Raymond.
\newblock Dynamic {E}rd{\H{o}}s-{P}{\'{o}}sa listing.
\newblock Available at
  \href{https://perso.limos.fr/~jfraymon/Erd%C5%91s-P%C3%B3sa/}{{\ttfamily
  https://perso.limos.fr/\break\string~jfraymon/Erdős-Pósa/}}.

\bibitem{R18}
J.~Raymond.
\newblock Hitting minors, subdivisions, and immersions in tournaments.
\newblock {\em Discrete Mathematics \& Theoretical Computer Science}, 20(1),
  2018.

\bibitem{RaymondT17}
J.~Raymond and D.~M. Thilikos.
\newblock Recent techniques and results on the {E}rd{\H{o}}s-{P}{\'{o}}sa
  property.
\newblock {\em Discret. Appl. Math.}, 231:25--43, 2017.

\bibitem{ReedRST96}
B.~A. Reed, N.~Robertson, P.~D. Seymour, and R.~Thomas.
\newblock Packing directed circuits.
\newblock {\em Combinatorica}, 16(4):535--554, 1996.

\bibitem{RobertsonS86}
N.~Robertson and P.~D. Seymour.
\newblock Graph minors. {V}. {E}xcluding a planar graph.
\newblock {\em Journal of Combinatorial Theory, Series {B}}, 41(1):92--114,
  1986.

\end{thebibliography}

\end{document}